\documentclass[10pt,reqno]{amsart}
\usepackage{amscd}
\usepackage{amsfonts}
\usepackage{amsmath}
\usepackage{amssymb}
\usepackage{amsthm}
\usepackage{enumitem}
\usepackage{float}
\usepackage{mathtools}
\usepackage{mathrsfs}
\usepackage{needspace}
\usepackage{setspace}
\usepackage{tikz}
\usepackage{version}

\usepackage{hyperref}

\onehalfspace

\theoremstyle{plain}
\newtheorem{lemma}{Lemma}[section]
  \newtheorem{proposition}[lemma]{Proposition}
  \newtheorem{corollary}[lemma]{Corollary}
  \newtheorem{theorem}[lemma]{Theorem}
  \newtheorem*{theorem*}{Theorem}

  \newtheorem*{fact*}{Fact}
  \newtheorem*{claim*}{Claim}

  \newtheorem{conjecture}[lemma]{Conjecture}
  \newtheorem{slogan}[lemma]{Slogan}

\theoremstyle{definition}
  \newtheorem{definition}[lemma]{Definition}

\theoremstyle{remark}
  \newtheorem{remark}[lemma]{Remark}

\setenumerate{label=\textup{(}\roman*\textup{)}}

\makeatletter
\newsavebox{\@brx}
\newcommand{\llangle}[1][]{\savebox{\@brx}{\(\m@th{#1\langle}\)}%
  \mathopen{\copy\@brx\kern-0.5\wd\@brx\usebox{\@brx}}}
\newcommand{\rrangle}[1][]{\savebox{\@brx}{\(\m@th{#1\rangle}\)}%
  \mathclose{\copy\@brx\kern-0.5\wd\@brx\usebox{\@brx}}}
\makeatother

\title{Finding a low-dimensional piece of a set of integers}
\author{Freddie Manners}
\address{Freddie Manners, Mathematical Institute, University of Oxford}
\email{manners@maths.ox.ac.uk}

\begin{document}

\newcommand{\eps}[0]{\varepsilon}

\newcommand{\AAA}[0]{\mathbb{A}}
\newcommand{\CC}[0]{\mathbb{C}}
\newcommand{\EE}[0]{\mathbb{E}}
\newcommand{\FF}[0]{\mathbb{F}}
\newcommand{\NN}[0]{\mathbb{N}}
\newcommand{\PP}[0]{\mathbb{P}}
\newcommand{\QQ}[0]{\mathbb{Q}}
\newcommand{\RR}[0]{\mathbb{R}}
\newcommand{\TT}[0]{\mathbb{T}}
\newcommand{\ZZ}[0]{\mathbb{Z}}

\newcommand{\cA}[0]{\mathscr{A}}
\newcommand{\cB}[0]{\mathscr{B}}
\newcommand{\cC}[0]{\mathscr{C}}
\newcommand{\cD}[0]{\mathscr{D}}
\newcommand{\cE}[0]{\mathscr{E}}
\newcommand{\cF}[0]{\mathscr{F}}
\newcommand{\cH}[0]{\mathscr{H}}
\newcommand{\cG}[0]{\mathscr{G}}
\newcommand{\cK}[0]{\mathscr{K}}
\newcommand{\cM}[0]{\mathscr{M}}
\newcommand{\cN}[0]{\mathscr{N}}
\newcommand{\cP}[0]{\mathscr{P}}
\newcommand{\cR}[0]{\mathscr{S}}
\newcommand{\cS}[0]{\mathscr{S}}
\newcommand{\cT}[0]{\mathscr{S}}
\newcommand{\cU}[0]{\mathscr{U}}
\newcommand{\cW}[0]{\mathscr{W}}
\newcommand{\cX}[0]{\mathscr{X}}
\newcommand{\cY}[0]{\mathscr{Y}}
\newcommand{\cZ}[0]{\mathscr{Z}}

\newcommand{\fg}[0]{\mathfrak{g}}
\newcommand{\fk}[0]{\mathfrak{k}}
\newcommand{\fZ}[0]{\mathfrak{Z}}

\newcommand{\bmu}[0]{\boldsymbol\mu}

\newcommand{\AUT}[0]{\mathbf{Aut}}
\newcommand{\Aut}[0]{\operatorname{Aut}}
\newcommand{\Frob}[0]{\operatorname{Frob}}
\newcommand{\GI}[0]{\operatorname{GI}}
\newcommand{\HK}[0]{\operatorname{HK}}
\newcommand{\HOM}[0]{\mathbf{Hom}}
\newcommand{\Hom}[0]{\operatorname{Hom}}
\newcommand{\Ind}[0]{\operatorname{Ind}}
\newcommand{\Lip}[0]{\operatorname{Lip}}
\newcommand{\LHS}[0]{\operatorname{LHS}}
\newcommand{\RHS}[0]{\operatorname{RHS}}
\newcommand{\Sub}[0]{\operatorname{Sub}}
\newcommand{\id}[0]{\operatorname{id}}
\newcommand{\image}[0]{\operatorname{Im}}
\newcommand{\poly}[0]{\operatorname{poly}}
\newcommand{\trace}[0]{\operatorname{Tr}}
\newcommand{\sig}[0]{\ensuremath{\tilde{\cS}}}
\newcommand{\psig}[0]{\ensuremath{\cP\tilde{\cS}}}
\newcommand{\metap}[0]{\operatorname{Mp}}
\newcommand{\symp}[0]{\operatorname{Sp}}
\newcommand{\dist}[0]{\operatorname{dist}}
\newcommand{\stab}[0]{\operatorname{Stab}}
\newcommand{\HCF}[0]{\operatorname{hcf}}
\newcommand{\LCM}[0]{\operatorname{lcm}}
\newcommand{\SL}[0]{\operatorname{SL}}
\newcommand{\GL}[0]{\operatorname{GL}}
\newcommand{\rk}[0]{\operatorname{rk}}
\newcommand{\sgn}[0]{\operatorname{sgn}}
\newcommand{\uag}[0]{\operatorname{UAG}}
\newcommand{\freiman}[0]{Fre\u{\i}man}
\newcommand{\tf}[0]{\operatorname{tf}}
\newcommand{\ev}[0]{\operatorname{ev}}
\newcommand{\bg}[0]{\operatorname{big}}
\newcommand{\sml}[0]{\operatorname{sml}}

\newcommand{\Conv}[0]{\mathop{\scalebox{1.5}{\raisebox{-0.2ex}{$\ast$}}}}
\newcommand{\bs}[0]{\backslash}

\newcommand{\heis}[3]{ \left(\begin{smallmatrix} 1 & \hfill #1 & \hfill #3 \\ 0 & \hfill 1 & \hfill #2 \\ 0 & \hfill 0 & \hfill 1 \end{smallmatrix}\right)  }

\newcommand{\uppar}[1]{\textup{(}#1\textup{)}}

\newcommand{\pol}[0]{\l}

\setcounter{tocdepth}{1}

\maketitle

\begin{abstract}
  We show that a finite set of integers $A \subseteq \ZZ$ with $|A+A| \le K |A|$ contains a large piece $X \subseteq A$ with \freiman{} dimension $O(\log K)$, where large means $|A|/|X| \ll \exp(O(\log^2 K))$.  This can be thought of as a major quantitative improvement on \freiman{}'s dimension lemma; or as a ``weak'' \freiman{}--Ruzsa theorem with almost polynomial bounds.

  The methods used, centered around an ``additive energy increment strategy'', differ from the usual tools in this area and may have further potential.

  Most of our argument takes place over $\FF_2^n$, which is itself curious.  There is a possibility that the above bounds could be improved, assuming sufficiently strong results in the spirit of the Polynomial \freiman{}--Ruzsa Conjecture over finite fields.
\end{abstract}

\tableofcontents{}

\section{Introduction}
\subsection{Motivation and statements}

Our primary motivation comes from considering \freiman{}-type theorems for subsets of the integers, of the following sort.

\begin{theorem}
  \label{thm:fr}
  Suppose $A \subseteq \ZZ$ and $|A+A| \le K |A|$.  Then there exists a generalized arithmetic progression $P$ of rank at most $C_1(K)$, such that $|P| \le C_2(K) |A|$ and $|A \cap P| \ge |A| / C_3(K) $, for some quantities $C_1, C_2, C_3$ depending only on $K$.
\end{theorem}

If one insists that $A \subseteq P$ then one cannot do better than $C_1(K) = K^{O(1)}$ and $C_2(K), C_3(K) = \exp\left(K^{O(1)}\right)$; such a result is due to \freiman{} \cite{freiman-bilu, freiman}.  If, as stated, we allow just a large subset of $A$ to be contained in $P$, the current overall best bounds in Theorem \ref{thm:fr} are due to Sanders \cite{sanders}, and are of the form $C_1(K) = O(\log^4 K)$, $C_2(K), C_3(K) = \exp(O(\log^4 K))$.

It is also possible to control the rank $C_1(K)$ optimally by $O(\log K)$, at the expense of poorer bounds on $C_2$ and $C_3$; see e.g.~\cite{bilu-1, bilu-2, freiman-bilu, green-tao-freiman-bilu}.

It has been widely conjectured that one can take $C_1(K) = O(\log K)$ while also insisting that $\log(C_2(K))$ and $\log(C_3(K))$ are both $O\big(\log^{1 + o(1)} K\big)$ (and indeed sharper conjectures could be made).

This paper has nothing to add to the state of the art in Theorem \ref{thm:fr}.  Instead, we consider a model problem that asks for significantly less structural information than Theorem \ref{thm:fr}.  To state this, we briefly recall the notion of \freiman{} dimension.

\begin{definition}
  Let $G, G'$ be abelian groups and $A \subseteq G$, $A' \subseteq G'$ finite subsets.  We say a function $\phi \colon A \to A'$ is a \emph{\freiman{} homomorphism} if $\phi(x) + \phi(w) = \phi(y) + \phi(z)$ whenever $x,y,z,w \in A$ and $x+w=y+z$.  We say it is a \emph{\freiman{} isomorphism} if it has an inverse that is a \freiman{} homomorphism.

  Now let $A \subseteq \ZZ$.  The \emph{\freiman{} dimension} of $A$, denoted $\dim(A)$, is the largest positive integer $d$ such there exists $A' \subseteq \ZZ^d$ which affinely generates $\ZZ^d$ (that is, $A'$ is not contained in any proper affine subspace of $\ZZ^d$), such that $A$ is \freiman{} isomorphic to $A'$.
\end{definition}

A key classical result about \freiman{} dimension, due to \freiman{} \cite{freiman} (or see e.g.~\cite[Lemma 5.13]{tao-vu}), is the following.

\begin{theorem}[{\freiman{} dimension lemma}]
  \label{thm:freiman-dim}
  Let $A \subseteq \ZZ$ be such that $|A+A| \le K |A|$.  Then $\dim(A) \le K - 1 + O(K^2/|A|)$.
\end{theorem}
In fact a slightly sharper result is proven, but the details of this need not concern us currently.

There is at least some connection between \freiman{} dimension and bounds in Theorem \ref{thm:fr}.  Specifically, Chang's proof of \freiman{}'s theorem \cite{chang} shows that one may take the rank of the progression $P$ to be  \emph{precisely} the \freiman{} dimension $\dim(A)$; the same bound is obtained by the proof in \cite[Chapter 5.6]{tao-vu}.

The bound in Theorem \ref{thm:freiman-dim} is essentially sharp: for instance, consider a set such as $\{a + 100^b N \colon 1 \le a \le N,\, 1 \le b \le k \}$ which has doubling about $(k+1)$ and \freiman{} dimension about $k$.  However, examples with very large \freiman{} dimension appear all to look like a union of a few structured pieces as in this case, each of which individually has much smaller \freiman{} dimension.  So, it is reasonable to expect a much stronger bound to hold if we are allowed to pass to a large \emph{subset} of $A$.  Note this is entirely analogous to the situation in Theorem \ref{thm:fr}.

Hence, the following is -- at least in some respects -- a model problem on the road towards polynomial bounds in Theorem \ref{thm:fr}.

\begin{conjecture}
  \label{conj:main}
  Suppose $A \subseteq \ZZ$ is a finite set and $|A+A| \le K |A|$.  Then we can find a subset $X \subseteq A$ with $|X| \ge |A| / K^{O(1)}$ and $\dim(X) = O(\log K)$.
\end{conjecture}

This is equivalent to what is called the ``Weak Polynomial \freiman{}--Ruzsa Conjecture'' in \cite{chang-pfr}.  It is a consequence of any sufficiently strong analogue of Theorem \ref{thm:fr} over $\ZZ^d$ (for all $d$).  It has some applications in its own right; e.g.~\cite[Corollary 2]{chang-pfr} uses it to deduce a weak version of the Erd\H{o}s--Szemer\'edi sum-product conjecture.

%

In this paper, we fail to prove Conjecture \ref{conj:main}, but obtain the following near miss as our main result.

\begin{theorem}
  \label{thm:main-theorem}
  Let $A \subseteq \ZZ$ be a finite set with $E(A) \ge |A|^3 / K$.  Then we can find a subset $X \subseteq A$ such that $\dim(X) = O(\log K)$ and $|A|/|X| \ll \exp(O(\log^2 K))$.
\end{theorem}

Note this is stated in terms of the additive energy
\[
  E(A) = \# \{ x,y,z,w \in A \colon x - y = z - w \}
\]
of the set rather than its doubling constant.  Since a set with small doubling has large energy, the energy formulation is at least as strong. In fact two formulations are essentially equivalent, by invoking the Balog--Szemer\'edi--Gowers theorem. However, the proof of Theorem \ref{thm:main-theorem} does not use this result, and works with energy directly.

As we have stated above, a large part of the interest in Theorem \ref{thm:main-theorem} from our perspective is contained in some of the techniques that are used in the proof, which appear to be novel and not terribly closely related to the existing literature in this area.\footnote{We have since been informed [Ben Green, personal communication] of previous unpublished work of Gowers which includes some closely related ideas.}  In a sense, Theorem \ref{thm:main-theorem} could be thought of as an application of these techniques rather than the primary motivation behind them.

In particular, the tools used are essentially disjoint from those appearing in Sanders' work \cite{sanders}.  Also, we spend most of our time working over the group $\FF_2^n$, which is unconventional given the result itself concerns subsets of $\ZZ$.

Although it is not our real goal, we state one application of our main result.  By combining Theorem \ref{thm:main-theorem} with the argument in \cite[Chapter 5.6]{tao-vu} used as a black box, we obtain the following version of Theorem \ref{thm:fr} as an easy consequence.

\begin{corollary}
  In Theorem \ref{thm:fr}, we may take $C_1(K) = O(\log K)$ and both $C_2(K), C_3(K)$ to be $O(\exp(\exp(O(\log^2 K))))$.
\end{corollary}

The bound on $C_2, C_3$ is not competitive compared to e.g.~\cite{green-tao-freiman-bilu} (although there is perhaps scope for optimizing this argument); we mention this corollary only because our derivation differs substantially from those arguments.

\subsection{Outline of the proof}

Our general strategy is one of \emph{additive energy increment}.  That is, we argue that if $\dim(A)$ is substantially larger than $\log K$, we can find a large subset $A_1 \subseteq A$ whose (normalized) additive energy is noticeably larger than that of $A$.  Since normalized energy (which we have not yet defined) is bounded above by $1$, this process must terminate after relatively few steps.

To achieve this, we follow the general philosophy of \cite{green-tao-freiman-bilu} or \cite{fernando}, which both argue along the following lines.  We may assume that, after a \freiman{} isomorphism, we have identified $A$ with a subset of $\ZZ^d$.
\begin{itemize}
  \item If $A$ were in fact a measurable subset of $\RR^d$, it would have to have doubling exponential in $d$ (by the Brunn--Minkowski inequality) or additive energy exponentially small in $d$ (by the sharp form of Young's inequality due to Beckner \cite{beckner} and Brascamp and Lieb \cite{brascamp-lieb}).
  \item Since it does not, we deduce that $A$ is not properly $d$-dimensional in some sense.
  \item We use this information to make progress by passing to a large subset of $A$.
\end{itemize}

In place of the Brunn--Minkowski or sharp Young inequalities, at the heart of our argument is a variant of the sharp Young inequality for ``$U^3$ energy'' (which we will define).  This has the advantage of having a dramatically simpler proof than the sharp Young inequality itself.  We can then transfer this argument from $\RR^d$ to $\ZZ^d$, and while the same conclusion cannot always hold in that setting, if it fails, the simplicity of the original proof translates into very strong structural information about our set.

Specifically, we will be able to show that if $A \subseteq \ZZ^d$ has normalized $U^3$ energy much bigger than $2^{-d}$, then the set must be very unevenly distributed modulo $2$.

We then seek to use this knowledge to obtain an energy increment.  The analysis required for this takes place largely over the finite vector space $\FF_2^d$.  This rather unexpected feature suggests the intriguing possibility that strong conjectures over $\FF_2^n$ (in the spirit of the Polynomial Fre\u{\i}man--Ruzsa conjecture over $\FF_2^n$) could imply strong results over the integers, and specifically further strengthenings of Theorem \ref{thm:main-theorem}.

\subsection{Layout of the paper}

In Section \ref{sec:reduction-mod-2}, we introduce the notion of $U^3$ energy, and explore the variant of the sharp Young inequality described above. We also prove a result stating that if $A \subseteq \ZZ^d$ fails to have small energy then it has a very uneven distribution modulo $2$.

In Section \ref{sec:energy-increment}, we deal with the additive energy increment argument and complete the proof of Theorem \ref{thm:main-theorem}.

At some point, we will need a useful inequality relating $U^3$ energy to usual additive energy, due to Shkredov \cite[Proposition 35]{shkredov}.  For interest and completeness, we give a self-contained alternative proof using entropy in Appendix \ref{sec:u3-inequality}.  This discussion is more or less orthogonal to the rest of the paper.

Finally, in Appendix \ref{app:doubling}, we describe -- for general interest only -- a closely related result to the material of Section \ref{sec:reduction-mod-2}, showing that sets of small doubling in $\ZZ^d$ occupy few residue classes modulo $2$.  Although we do not make use of this result in the main argument, we include it mainly because its proof is very different to the methods of Section \ref{sec:reduction-mod-2}, and it is curious that such dissimilar arguments should obtain such similar and yet -- to the author at least -- unfamiliar conclusions. 

\subsection{Notation}

We use the standard conventions $X = O(Y)$ to mean $X \le C Y$ for some absolute constant $C$, and $X \ll Y$ to mean $X = O(Y)$.

We have frequently adopted the slightly non-standard habit of using integral signs to denote sums or averages over discrete or even finite groups.  Hence for instance given a function $f \colon \ZZ/N\ZZ \to \RR$ we might write
\[
  \|f\|_2 = \left( \int_{x} |f(x)|^2 \right)^{1/2}
\]
even though the domain is finite.  This allows us to be agnostic in our notation concerning the choice of normalization of the measure (e.g.~counting measure or probability measure); i.e.~the above could denote either a sum or an average, depending on context.  Often, any \emph{consistent} choice will be valid, in that we frequently manipulate quantities or equalities that are unaffected by the choice of normalization. 

On a few occasions, though, it will be genuinely important that counting measure is used on the groups $\ZZ^d$ and $\FF_2^d$, and that $L^p$- and $U^k$-norms on these groups are defined accordingly.  We therefore adopt the convention that counting measure is used throughout for these groups, and will flag up when this actually matters.

Various symbols, notably $A'$, $\eta$ and $K$, will retain their meanings throughout large parts of the argument.  The original definitions can be found in the discussion at the start of Section \ref{sec:energy-increment} (but before Section \ref{subsec:casea}).

For a set $X$, we will sometimes denote its indicator function also by $X$ rather than $1_X$, to aid legibility.

\section{Reduction modulo $2$ in $\ZZ^d$ and $U^3$ energy}
\label{sec:reduction-mod-2}

For reasons that are not entirely clear, our arguments will necessarily be concerned not with the usual notion of additive energy
\[
  E(A) = \# \{ x,y,z,w \in A \colon x - y = z - w \} = \|1_A\|_{U^2}^4
\]
but with the much less standard notion of \emph{$U^3$ energy}, which we denote
\[
  E_{U^3}(A) = \|1_A\|_{U^3}^8
\]
where in both cases we assume counting measure in the definition of our $U^k$-norms.  Concretely we therefore have
\begin{align*}
  E_{U^3}(A) = \sum_{x,h_1,h_2,h_3 \in \ZZ} A(x) &A(x+h_1) A(x+h_2) A(x+h_1+h_2) A(x+h_3) \\& A(x+h_1+h_3) A(x+h_2+h_3) A(x+h_1+h_2+h_3) \ . 
\end{align*}
As usual, there is a multilinear variant of the $U^3$ energy, which we denote by
\begin{align*}
  \llangle A_1, \dots, A_8 \rrangle_{U^3} = \sum_{x,h_1,h_2,h_3 \in \ZZ} A_1(x) &A_2(x+h_1) A_3(x+h_2) A_4(x+h_1+h_2) A_5(x+h_3) \\& A_6(x+h_1+h_3) A_7(x+h_2+h_3) A_8(x+h_1+h_2+h_3)
\end{align*}
for sets $A_1, \dots, A_8$, and which is the same as the $U^3$ Gowers inner product.

The reader may reasonably be concerned that the $U^3$-norm of a set might be a genuinely different quantity than its energy, or $U^2$-norm: certainly these are radically different notions when applied to complex-valued functions.  However, it turns out that in the setting of indicator functions, or more generally non-negative real functions, these are comparable measures of additive structure.

More precisely, we state the following.

\begin{proposition}
  \label{prop:u3-energy-inequalities}
  If $Z$ is any abelian group and $f \colon Z \to \RR_{\ge 0}$ a non-negative finitely supported function, then
  \[
    \| f \|_{U^3} \ge \| f \|_{U^2}^2 / \|f\|_1 \ .
  \]
  In the other direction,
  \[
    \| f \|_{U^3} \le \| f \|_{U^2}^{1/2} \| f\|_4^{1/2} \ .
  \]
  We remark that these inequalities are independent of the choice of measure on $Z$ used to define these norms.
\end{proposition}
\begin{proof}
  The second inequality is straightforward: in the expression
  \[
    \|f\|_{U^3}^8 = \int_{x,y,h_1,h_2} f(x) f(x+h_1) f(x+h_2) f(x+h_1+h_2) f(y) f(y+h_1) f(y+h_2) f(y+h_1+h_2)
  \]
  we apply the AM--GM inequality $a b c d \le (a^4 + b^4 + c^4 + d^4)/4$ pointwise to the last four terms.

  The first is significantly less trivial.  In the case that $f$ is the indicator function of a set (which is all we need, and from which the general case could be deduced), this result is due to Shkredov \cite[Proposition 35]{shkredov}; his proof uses a Loomis--Whitney-type result attributed to Bollob\'as and Thomason \cite{bt95}, and the tensor power trick.

  In Appendix \ref{sec:u3-inequality}, for interest, we give a self-contained alternative (but morally similar) proof using entropy.
\end{proof}

We recall that Young's inequality implies the bound $\|f\|_{U^2} \le \|f\|_{4/3}$.  We begin by recalling a version of this for the $U^3$-norm.

\begin{proposition}
  Let $N$ be an odd prime and suppose $f \colon \ZZ/N\ZZ \to \CC$.  We have $\|f\|_{U^3} \le \|f\|_2$.
\end{proposition}
This is a special case of a general result for all $U^k$-norms; see \cite{tao-eisner}.  The usual proof is iterated application of Young's inequality for convolutions.  We note that in this case there is an even easier proof.  
\begin{proof}
  We consider
  \begin{align*}
    \|f\|_{U^3}^8 = \int_{x,h_1,h_2,h_3} & \left( f(x) f(x+h_1+h_2) f(x+h_1+h_3) f(x+h_2+h_3) \right) \\
    &\ \  \left( \overline{f(x+h_1) f(x+h_2) f(x+h_3) f(x+h_1+h_2+h_3)} \right)
  \end{align*}
  and apply the Cauchy--Schwarz inequality to obtain
  \begin{align*}
    \|f\|_{U^3}^8 \le &\left(\int_{x,h_1,h_2,h_3} |f(x)|^2 |f(x+h_1+h_2)|^2 |f(x+h_1+h_3)|^2 |f(x+h_2+h_3)|^2\right)^{1/2} \\
                      &\left(\int_{x,h_1,h_2,h_3} |f(x+h_1+h_2+h_3)|^2 |f(x+h_1)|^2 |f(x+h_2)|^2 |f(x+h_3)|^2\right)^{1/2} \\
                      &= \int_{x,h_1,h_2,h_3} |f(x)|^2 |f(x+h_1+h_2)|^2 |f(x+h_1+h_3)|^2 |f(x+h_2+h_3)|^2
  \end{align*}
  where we have used the change of variables
  \begin{align}
    \label{eqn:easy-cov}
    \begin{split}
    x' &= x + h_1 + h_2 + h_3 \\
    h_1' &= -h_1 \\
    h_2' &= -h_2 \\
    h_3' &= -h_3
    \end{split}
  \end{align}
  to show equality in the second step.  But by a further change of variables:
  \begin{align}
    \label{eqn:cov}
    \begin{split}
    r &= x \\
    s &= x+h_1+h_2 \\
    t &= x+h_1+h_3 \\
    u &= x+h_2+h_3
    \end{split}
  \end{align}
  this integral is precisely
  \[
    \left(\int_y |f(y)|^2 \right)^4
  \]
  as required.
\end{proof}

Over the domain $\RR^d$ rather than $\ZZ/N\ZZ$, there is a sharp form of Young's inequality due to Beckner \cite{beckner} and Brascamp--Lieb \cite{brascamp-lieb}, which in particular implies that $\|f\|_{U^2} \le C^d \|f\|_{4/3}$ for some explicit constant $C < 1$.  Again, we consider an analogue for the $U^3$-norm.

\begin{proposition}
  Let $f \colon \RR^d \to \CC$ be a well-behaved function \uppar{say, continuous and compactly supported}.  We have $\|f\|_{U^3} \le (1/2)^{d/8} \|f\|_2$.
\end{proposition}
As before, this can be deduced by iterated application of the \emph{sharp} Young inequality -- again see \cite{tao-eisner} for such a proof.  Alternatively, it is a special case of the more general Brascamp--Lieb inequality \cite{brascamp-lieb}.

However, as remarked in their original paper \cite{brascamp-lieb}, some special cases of the Brascamp--Lieb inequality have unusually easy proofs.  It turns out this is one such instance, and that we have essentially already seen the proof.
\begin{proof}
  We apply the Cauchy--Schwarz argument as above to deduce that
  \[
    \|f\|_{U^3}^8 \le \int_{x,h_1,h_2,h_3 \in \RR^d} |f(x)|^2 |f(x+h_1+h_2)|^2 |f(x+h_1+h_3)|^2 |f(x+h_2+h_3)|^2 \ ,
  \]
  noting that the matrix associated to the change of variables \eqref{eqn:easy-cov} has determinant $-1$.
  When we apply the linear change of variables \eqref{eqn:cov}, however,
  we find that the corresponding matrix
  \begin{equation}
    \label{eq:cov-matrix}
    \begin{pmatrix} 1 & 0 & 0 & 0 \\ 1 & 1 & 1 & 0 \\ 1 & 1 & 0 & 1 \\ 1 & 0 & 1 & 1 \end{pmatrix}^{\otimes d}
  \end{equation}
  has determinant $(-2)^d$.  It follows that
  \[
    \|f\|_{U^3}^8 \le 2^{-d} \|f\|_2^8
  \]
  as required.
\end{proof}

Of course, we are ultimately interested not in functions on $\RR^d$ but in subsets of $\ZZ^d$.  Since the argument above is so elementary, it is straightforward to transfer it to this setting and see what we obtain.

\begin{proposition}
  \label{prop:zd-bl-inequality}
  If $f \colon \ZZ^d \to \CC$ then
  \[
    \|f\|_{U^3}^8 \le \int_{r,s,t,u} [r + s + t + u \equiv 0\ (\bmod\,2)]\  |f(r)|^2 |f(s)|^2 |f(t)|^2 |f(u)|^2 \ .
  \]
\end{proposition}
\begin{proof}
  We apply the same Cauchy--Schwarz argument as above, and again make the change of variables \eqref{eqn:cov}. This time, we observe that the image of the matrix \eqref{eq:cov-matrix} applied to $\ZZ^{4d}$ consists of precisely those tuples $(r,s,t,u) \in \ZZ^{4d}$ such that $r + s + t + u \in 2 \ZZ^d$, as required.  Indeed, one inclusion is clear, and equality of these two sets then follows from the determinant calculation above.
\end{proof}

The right hand side expression is a function only of what we shall define as
\begin{align*}
  \eta \colon (\ZZ/2\ZZ)^d &\to \RR_{\ge 0} \\
  \omega &\mapsto \left(\int_x [x \bmod 2 = \omega]\ |f(x)|^2 \right)^{1/2}
\end{align*}
i.e.~the distribution of the $L^2$ mass of $f$ modulo $2$.  Moreover, we see that if $\eta$ is a constant function, i.e., the mass of the original function $f$ was equidistributed over residue classes modulo $2$, then this expression is precisely $\|f\|_2^8 / 2^d$, just as we saw over $\RR^d$.

In other words, applying this to $f = 1_A$ for some finite set $A \subseteq \ZZ^d$, we deduce the following.
\begin{slogan}
  If $A \subseteq \ZZ^d$ is a finite set and the energy $E_{U^3}(A)$ is much bigger than $2^{-d} |A|^4$, then $A$ must be highly non-equidistributed modulo $2$.
\end{slogan}

\begin{remark}
  One possible interpretation for this rather bizarre-looking statement is that if $A$ were properly $d$-dimensional then its energy should be exponentially small in $d$; so the non-flat distribution modulo $2$ is a kind of witness to the fact that $A$ is not properly $d$-dimensional.
\end{remark}

We now make this observation more precise.

First, we note that -- since positive and negative signs are irrelevant modulo $2$ -- the expression appearing on the right hand side of Proposition \ref{prop:zd-bl-inequality} is precisely $\|\eta^2\|_{U^2}^4$ where $\eta$ is defined as above. Here, we stress the importance of using counting measure on $(\ZZ/2\ZZ)^d$ to define this norm.

Using the Young's inequality bound $\|h\|_{U^2} \le \|h\|_{4/3}$, we obtain the following.

\begin{proposition}
  \label{prop:mod-2-lp-clustering}
  Let $f \colon \ZZ^d \to \CC$ and let $\eta \colon (\ZZ/2\ZZ)^d \to \RR_{\ge 0}$ be defined as above.  Then
  \[
    \|f\|_{U^3} \le \|\eta\|_{8/3} \ .
  \]
\end{proposition}
\begin{proof}
  We know that $\|f\|_{U^3}^8 \le \|\eta^2\|_{U^2}^4$ and $\|\eta^2\|_{U^2} \le \|\eta^2\|_{4/3} = \|\eta\|_{8/3}^2$.  The result follows.
\end{proof}

It follows that if $\|f\|_{U^3} / \|f\|_2 = 1/K$ is large then so is $\|\eta\|_{8/3} / \|\eta\|_2$ (since $\|f\|_2 = \|\eta\|_2$).  Since $8/3$ is bigger than $2$, this tells us that $\eta$ must have a positive proportion of its $L^2$ mass on a just few points of $(\ZZ/N\ZZ)^d$.  Specifically, there is an $x \in (\ZZ/2\ZZ)^d$ such that  $\eta(x) \ge K^{-4} \|\eta\|_2$.

In the case that $f = 1_A$ is the indicator function of a set, this has exactly the same flavour as Proposition \ref{prop:few-values-mod-2}.

\section{Energy increment arguments}
\label{sec:energy-increment}

The key quantity in the remainder of the argument is what we call the \emph{normalized $U^3$ energy} of a set:
\[
  \cE(A) = \left(E_{U^3}(A) / |A|^4\right)^{1/8} = \|A\|_{U^3} / \|A\|_2 \ .
\]
Clearly the same quantity also makes sense for any (non-negative) finitely supported function $f$; i.e.~we define
\[
\cE(f) = \|f\|_{U^3} / \|f\|_2
\]
for any discrete abelian group $Z$ and any finitely supported $f \colon Z \to \RR_{\ge 0}$.  We remark that this quantity does not depend on the normalization of Haar measure on $Z$.

Our detailed strategy to use the results of the previous section to prove Theorem \ref{thm:main-theorem} is as follows.

\begin{itemize}
  \item Let $A \subseteq \ZZ$ be a finite set with large normalized $U^3$ energy, $\cE(A) = 1 / K$.
  \item Write $d = \dim(A)$. If $d = O(\log K)$, we are done.  If not, we proceed as follows.
  \item By the definition of \freiman{} dimension, there is a \freiman{} isomorphic copy $A' \subseteq \ZZ^d$ of $A$, which affinely generates $\ZZ^d$.  Note (see Proposition \ref{prop:u3-freiman}) that \freiman{} isomorphism preserves quantities such as the $U^3$ energy of a set.
  \item We now inspect the function $\eta \colon (\ZZ/2\ZZ)^d \to \RR_{\ge 0}$ defined in the previous section for $f = 1_{A'}$, i.e., the $L^2$ distribution of $A'$ modulo $2$.  By Proposition \ref{prop:mod-2-lp-clustering}, this has a positive proportion of its $L^2$ mass concentrated on only a few values.
  \item We will now use this information to obtain an additive energy increment upon passing to a subset of $A'$.  Specifically, we will find a subset $S \subseteq (\ZZ/2\ZZ)^d$, define
    \[
      A'_S := \{ a \in A' \colon a \bmod 2 \in S \} \subseteq A'
    \]
    and show that
    \[
      \cE(A'_S) \ge \alpha\, \cE(A')
    \]
    and
    \[
      |A'| / |A'_S| \le \beta
    \]
    for suitable parameters $\alpha > 1$ and $\beta$.

    Finally, we pass to the corresponding set $A_S \subseteq A$, where the exact same conclusions hold.  We then iterate this whole argument from the beginning on $A_S$.
\end{itemize}

The aim is of course to make the parameter $\alpha$ as large as possible while keeping $\beta$ fairly small.  We might hope to have $\alpha \ge 1.01$ and $\beta = O(K^{O(1)})$.  In this case, since the normalized energy increases by a factor of $\alpha$ at each stage and is bounded above by $1$, the algorithm can proceed for at most $O(\log K)$ iterations, meaning the size of the final set relative to $A$ is about $K^{-O(\log K)} = \exp(-O(\log^2 K))$ as required.
    
Our procedure for finding this energy increment is slightly involved, and splits into several cases.  We stress that there is no one obvious way to proceed in this phase of the argument, and there may well be more elegant ways to combine the various options than we have managed to find. 
    
The three cases we will consider are as follows.
\begin{enumerate}[label=(\Alph*)]
  \item We know that a positive proportion of the $L^2$ mass of $\eta$ is clustered on a few points where $\eta$ is large.  Suppose that a significant amount of $L^2$ mass can also be found on points where $\eta$ is \emph{small}; currently this cannot be ruled out.

    Then, by discarding these points -- which by Proposition \ref{prop:mod-2-lp-clustering} has very little effect on $\|A'\|_{U^3}$ -- we can obtain a fairly respectable energy increment at very little cost.  In the above notation, we set
    \[
      S = \{ x \in (\ZZ/2\ZZ)^d \colon \eta(x) \text{ is large} \} \ .
    \]

    So, we may henceforth assume that almost all of the $L^2$ mass of $\eta(x)$ is found at points where $\eta(x) \gg K^{-C}$.

  \item The mainstay of the strategy is to restrict to a single fiber modulo $2$.  In the above notation we take $S \subseteq (\ZZ/2\ZZ)^d$ to be a single point.  By the previous case, this involves a density penalty of $O(K^C)$.

    We exploit the fact that the energy of these fibers is, on average (in an appropriate sense), at least the energy of the whole set.  Moreover, it will usually be slightly larger, \emph{unless} the function $\eta$ has normalized energy close to the maximum, i.e., $\cE(\eta) \ge 0.99$.

  \item The final case handles this possibility that $\eta$ has close to maximal energy.
    
    Fortunately, in this case there is a structure theorem, that asserts that $\eta$ is close to being the indicator function of a coset of a subgroup of $(\ZZ/2\ZZ)^d$, correctly normalized.  Since $\eta$ has small support compared to the whole group $(\ZZ/2\ZZ)^d$, this has to be a proper subgroup.

    However, $\eta$ cannot be entirely supported on this coset, since by assumption $A'$ affinely generates $\ZZ^d$.  So, there are a few stray points of $\eta$ lying outside.

    The strategy is then (usually) to delete these extraneous points.  One can argue that this obtains an energy increment: possibly a very modest one, but at a similarly modest cost in density.
\end{enumerate}

In the main case, Case B, we increase the normalized energy by a factor of at least about $1.001$ (say) and lose $O(K^{O(1)})$ in the density. As we remark above, it follows this case can occur only $O(\log K)$ times, and so the density loss is of the form $\exp(O(\log^2 K))$.

The analysis of Cases A and C is rather more delicate: there is no lower bound on the energy increase, and so each might occur an unbounded number of times.  However, in the final evaluation we will see that these cases are in fact very efficient compared to Case B, achieving a polynomial trade-off between energy and density.

\begin{remark}
  The unwelcome $\log^2 K$ comes from the inherently wasteful practice of restricting to precisely \emph{one} fiber modulo $2$ in Case B, which loses a factor of $K^{O(1)}$ density, in exchange for relatively small energy increment.

  One might reasonably conjecture that a better energy increment can always be obtained using the more refined practice of passing to a coset as in Case C.  However, our analysis does not show this.  Moreover, such a conjecture seems likely to be roughly as hard as the Polynomial Fre\u{\i}man--Ruzsa Conjecture over $\FF_2^n$.
  
  As mentioned above, this does leave open the possibility that a sufficiently good understanding of quantitative structural theory over $\FF_2^n$ could imply significantly improved results over $\ZZ$ using these methods.
\end{remark}

The remainder of this section will be concerned with making the above sketch precise, and in particular with the analysis of these three cases.  This argument is rather technical (although rarely actually difficult) and has little conceptual content beyond what we have already said.

\vspace{\baselineskip}

Finally before we proceed, we shall fix some further notation.  As well as
\[
  A'_S = \{ a \in A' \colon a \bmod 2 \in S \}
\]
for $S \subseteq (\ZZ/2\ZZ)^d$, we will write
\[
  A'_\omega = \{ a \in A' \colon a \bmod 2 = \omega \}
\]
for the degenerate case where $S = \{\omega\}$ is a single point $\omega \in (\ZZ/2\ZZ)^d$. So, for instance, $\eta(\omega) = \|A'_{\omega}\|_2$, and we continue to use this symbol in what follows.

As above, the constant $K \ge 1$ will always be defined by $\cE(A') = 1 / K$.

\subsection{Case A: discarding points where $\eta$ is small}
\label{subsec:casea}

Recall we are trying to show that either we can obtain a reasonable energy increment, or we may assume that, up to a negligible error, $\eta$ is supported on large values.

We formalize this as follows.

\begin{proposition}
  \label{prop:discard-small-points}
  We use the notation from above.  Let $0 < R < 1$ be a parameter.  We have the following dichotomy: either $\eta$ can be decomposed as
  \[
    \eta = \eta_{\bg} + \eta_{\sml}
  \]
  where $\eta_{\bg}$, $\eta_{\sml}$ have disjoint supports, $\eta_{\bg}(x) / \|\eta\|_2 \ge 2^{-8} R^4 K^{-4}$ whenever it is non-zero and $\|\eta_{\sml}\|_2 \le R \|\eta\|_2$; or we can find a set $S \subseteq (\ZZ/2\ZZ)^d$ such that
  \[
    \frac{\cE(A'_S)}{\cE(A')} \ge (|A'| / |A'_S|)^{1/4} \ .
  \]
  In the latter case, we may further assume that $A'_S \ne A'$.
\end{proposition}

Observe that the energy increment we obtain here is a power of the density loss, which is acceptable and indeed much more efficient than we hope to achieve overall.

We first isolate the following lemma.

\begin{lemma}
  \label{lem:discard-small-points}
  We use all the notation from above.  Let
  \[
    S := \{ x \in (\ZZ/2\ZZ)^d \colon \eta(x) / \|\eta\|_2 \ge W \}
  \]
  for some parameter $W$.  Define $t := |A'_S| / |A'|$.  Then
  \[
    \frac{\cE(A'_S)}{\cE(A')} \ge t^{-1/2} \left(1 - W^{1/4} K \sqrt{1 - t} \right) \ .
  \]
\end{lemma}

\begin{proof}
  Write $f = 1_{A'}$, $f_{\bg} = 1_{A'_S}$ and $f_{\sml} = f - f_{\bg}$.  Similarly define $\eta_{\bg} = \eta \cdot 1_S$ and $\eta_{\sml} = \eta - \eta_{\bg}$.
  
  By Proposition \ref{prop:mod-2-lp-clustering} we have that
  \begin{align*}
    \|f_{\sml}\|_{U^3} &\le \| \eta_{\sml} \|_{8/3} \\
                       &\le \| \eta_{\sml} \|_\infty^{1/4} \| \eta_{\sml} \|_2^{3/4} \\
                       &\le W^{1/4} \| \eta_{\sml} \|_2 \ .
  \end{align*}
  By the triangle inequality, $\|f_{\bg}\|_{U^3} \ge \|f\|_{U^3} - \|f_{\sml}\|_{U^3}$ and so
  \begin{align*}
    \frac{\cE(f_{\bg})}{\cE(f)}
    &= \frac{\|f\|_2}{\|f_{\bg}\|_2} \frac{\|f_{\bg}\|_{U^3}}{\|f\|_{U^3}} \\
    &\ge t^{-1/2} \left(1 - \frac{W^{1/4} \|\eta_{\sml}\|_2}{\|f\|_{U^3}} \right) \\
    &= t^{-1/2} \left(1 - W^{1/4} K \sqrt{1 - t} \right)
  \end{align*}
  as required.
\end{proof}

\begin{proof}[Proof of Proposition \ref{prop:discard-small-points}]
  We set $W = 2^{-8} R^4 K^{-4}$ in Lemma \ref{lem:discard-small-points} and take $S$ to be the set from the statement of that lemma, i.e., $S$ is the set of points $x$ where $\eta(x) \ge W$.  Similarly, define $\eta_{\bg}$, $\eta_{\sml}$ as above.

  If the first case of the dichotomy does not hold with these choices, it must be that
  \[
    \|\eta_{\sml}\|_2 \ge R \|\eta\|_2
  \]
  since the other conditions hold by definition.  Using the notation $t = |A'_S| / |A'|$ as above and noting that
  \[
    (\|\eta_{\bg}\|_2 / \|\eta\|_2)^2 = t,\ \ (\|\eta_{\sml}\|_2 / \|\eta\|_2)^2 = 1 - t
  \]
  this says precisely that $(1 - t) \ge R^2$.

  In that case, we apply Lemma \ref{lem:discard-small-points} to deduce that
  \begin{align*}
    \frac{\cE(A'_S)}{\cE(A')} &\ge t^{-1/2} \left(1 - W^{1/4} K \sqrt{1 - t} \right) \\
                                                                   & = t^{-1/2} \left(1 - R \sqrt{1-t} / 4 \right) \\
                                                                   &\ge t^{-1/2} \left(1 - (1 - t)/4 \right) \\
                                                                   &\ge t^{-1/4}
  \end{align*}
  as required (since $3/4 + t/4 \ge t^{1/4}$ by the AM--GM inequality).
\end{proof}

\subsection{Case B: the small energy regime and passing to a fiber}

Our goal in this section is the following result.

\begin{proposition}
  \label{prop:small-energy}
  Again we continue with previous notation.  Fix a parameter $\eps > 0$.  Suppose that
  \[
    \cE(\eta) \le 1 - \eps
  \]
  Then either we may find a set $S \subseteq (\ZZ/2\ZZ)^d$ such that
  \[
    \frac{\cE(A'_S)}{\cE(A')} \ge (|A'| / |A'_S|)^{1/4}
  \]
  and $A'_S \ne A'$, as in Proposition \ref{prop:discard-small-points}; or there is some $\omega \in (\ZZ/2\ZZ)^d$ such that
  \[
    |A'| / |A'_{\omega}| \le 2^{24} K^{16} \eps^{-8}
  \]
  and
  \[
    \frac{\cE(A'_{\omega})}{\cE(A')} \ge 1 + \eps/2 \ .
  \]
\end{proposition}

First, we note the following inequality concerning the fibering properties of $U^3$ energy, which follows easily from the Gowers--Cauchy--Schwarz inequality for the $U^3$-norm.

\begin{lemma}
  \label{lem:gowers-cauchy-schwarz}
  Let $G$, $H$ be two abelian groups and $\phi \colon G \to H$ a surjective group homomorphism.  Let $f \colon G \to \RR_{\ge 0}$ be a finitely supported function.  For $h \in H$, write
  \[
    f_h(x) = f(x) \cdot 1_{\phi^{-1}(h)}(x) \ .
  \]
  Finally, define
  \begin{align*}
    \gamma \colon H &\to \RR_{\ge 0} \\
                  h &\mapsto \|f_h\|_{U^3} \ .
  \end{align*}
  Then $\|f\|_{U^3} \le \|\gamma\|_{U^3}$.
\end{lemma}
\begin{proof}
  Clearly
  \[
    \|f\|_{U^3}^8 = \left\| \sum_{h \in H} f_h \right\|_{U^3}^8
  \]
  and expanding this out we obtain
  \[
    \|f\|_{U^3}^8 = \sum_{h_1, \dots, h_8 \in H} \llangle f_{h_1}, \dots, f_{h_8} \rrangle_{U^3}
  \]
  (where we recall the definition of the Gowers inner product $\llangle \cdot, \dots, \cdot \rrangle_{U^3}$ from Section \ref{sec:reduction-mod-2}).
  If $(h_1, \dots, h_8)$ do \emph{not} form a $3$-dimensional parallelepiped in $H$, the corresponding Gowers inner product will vanish.  Hence
  \[
    \|f\|_{U^3}^8 = \sum_{h, r_1, r_2, r_3 \in H} \llangle f_{h}, f_{h + r_1}, \dots, f_{h + r_1 + r_2 + r_3} \rrangle_{U^3} \ .
  \]
  By the Gowers--Cauchy--Schwarz-type inequality for the $U^3$ inner product, i.e.:
  \[
    \llangle F_1, \dots, F_8 \rrangle_{U^3} \le \|F_1\|_{U^3} \dots \|F_8\|_{U^3}
  \]
  (see e.g.~\cite[Equation 11.6]{tao-vu}), we deduce that
  \[
    \|f\|_{U^3}^8 \le \sum_{h, r_1, r_2, r_3 \in H} \|f_h\|_{U^3} \|f_{h + r_1}\|_{U^3} \dots \|f_{h + r_1 + r_2 + r_3}\|_{U^3}
  \]
  and the right hand side is precisely $\|\gamma\|_{U^3}^8$, as required.
\end{proof}

For our current application, we will apply this when $G = \ZZ^d$, $H = (\ZZ/2\ZZ)^d$, $\phi$ is the obvious projection map and $f = 1_{A'}$.  In particular, we will now write
\begin{align*}
  \gamma \colon (\ZZ/2\ZZ)^d &\to \RR_{\ge 0} \\
  h &\mapsto \|A'_h\|_{U^3} \ .
\end{align*}

From this point, the proof of Proposition \ref{prop:small-energy} is essentially straightforward, but the details are rather involved.

\begin{proof}[Proof of Proposition \ref{prop:small-energy}]
  Let us first sketch the case where $\eta$ only takes large values on its support.  We have that
  \begin{align*}
    \|A'\|_{U^3} &\le \|\gamma \|_{U^3} \\
                 &\le \sup_{x \colon \eta(x) \ne 0} \{ \gamma(x) / \eta(x) \} \|\eta\|_{U^3} \\
                 &\le \sup_{x \colon \eta(x) \ne 0} \{ \gamma(x) / \eta(x) \} (1 - \eps) \|\eta\|_2
  \end{align*}
  where we have applied Lemma \ref{lem:gowers-cauchy-schwarz} and noted that if $\eta(x) = 0$ then necessarily $\gamma(x) = 0$.

  So this states precisely that there exists some $x \in (\ZZ/2\ZZ)^d$ with $\eta(x) \ne 0$ such that
  \[
    \gamma(x) / \eta(x) \ge (1 - \eps)^{-1} \|A'\|_{U^3} / \|A'\|_2
  \]
  and since $\gamma(x) = \|A'_{x}\|_{U^3}$, $\eta(x) = \|A'_{x}\|_2$ this is what we want -- provided we can be sure that $|A'|/|A'_x| = O(K^{O(1)})$ for any $x$ with $\eta(x) \ne 0$.

  By applying Proposition \ref{prop:discard-small-points}, we are free to assume this, up to a small error in the $L^2$ norm in some sense.  The technical work is now to rerun this argument, tolerating this further error term.
  
  Setting $R = \eps / 2 K$ in Proposition \ref{prop:discard-small-points}, either we have the energy increment conclusion as required, or we may decompose
  \[
    \eta = \eta_{\bg} + \eta_{\sml}
  \]
  where $\eta_{\bg}(x) \ge 2^{-12} \eps^4 K^{-8}\|\eta\|_2$ on its support $S \subseteq (\ZZ/2\ZZ)^d$, and $\|\eta_{\sml}\|_2 / \|\eta_2\| \le \eps / 2 K$.  We similarly decompose $\gamma = \gamma_{\bg} + \gamma_{\sml}$ where $\gamma_{\bg} = 1_S \cdot \gamma$.  Then arguing much as before,
  \begin{align*}
    \|A'\|_{U^3} &\le \|\gamma \|_{U^3} \\
                 &\le \|\gamma_{\bg}\|_{U^3} + \|\gamma_{\sml}\|_{U^3} \\
                 &\le \|\gamma_{\bg}\|_{U^3} + \|\gamma_{\sml}\|_2 \\
                 &\le \sup_{x \in S} \{ \gamma(x) / \eta(x) \} \|\eta_{\bg}\|_{U^3} + \|\eta_{\sml}\|_2 \\
                 &\le \sup_{x \in S} \{ \gamma(x) / \eta(x) \} (1 - \eps) \|\eta\|_2 + \|\eta_{\sml}\|_2
  \end{align*}
  where we have made use of the pointwise bound $\gamma(x) \le \eta(x)$ and the trivial bound $\|\eta_{\bg}\|_{U^3} \le \|\eta\|_{U^3}$. Hence, there exists an $x \in S$ such that
  \begin{align*}
    \frac{\gamma(x) / \eta(x)}{\cE(A')} &\ge (1 - \eps)^{-1} \left( 1 - \frac{\|\eta_{\sml}\|_2/\|\eta\|_2}{\|A'\|_{U^3} / \|A'\|_2} \right) \\
                                        &\ge (1 - \eps)^{-1} (1 - (\eps / 2 K) \cdot K) \ge 1 + \eps / 2
  \end{align*}
  as required.  Also, the condition
  \[
    |A'| / |A'_x| \le 2^{24} K^{16} \eps^{-8}
  \]
  is satisfied since $x \in S$ (and by the definition of $S$).
\end{proof}

\subsection{Case C: large energy and passing to a coset}

Recall that our final case handles the possibility that $\eta$ has very close to maximal normalized $U^3$ energy.  Specifically, we will prove the following.

\begin{proposition}
  \label{prop:99pc-energy-case}
  There exists some absolute constant $\delta > 0$ such that the following holds. Suppose $\cE(\eta) \ge 1 - \delta$, and $K \le 2^{d/8 - 4}$. Then there exists a set $S \subseteq (\ZZ/2\ZZ)^d$ such that
  \[
    \frac{\cE(A'_S)}{\cE(A')} \ge (|A'| / |A'_S|)^{1/8}
  \]
  and furthermore $A'_S \ne A'$.
\end{proposition}

Again we note that in this phase of the argument we get an energy increment which is polynomially large compared to the density penalty.

Our key ingredient is the following result.

\begin{proposition}
  \label{prop:99pc-structure-thm}
  For all $\eps > 0$ there exists $\delta > 0$ such that the following holds.  Let $G$ be any discrete abelian group and $f \colon G \to \RR_{\ge 0}$ a finitely supported function with $\|f\|_2 = 1$.  Suppose that $\|f\|_{U^3} \ge 1 - \delta$.

  Then there exists a finite subgroup $H \le G$ and an element $x \in G$ such that
  \[
    \left\| f - \frac{1}{|H|^{1/2}} 1_{x + H} \right\|_2 \le \eps \ .
  \]
\end{proposition}
\begin{proof}
  This is essentially a special case of \cite[Theorem 1.4]{tao-eisner}.  There, the authors consider functions which need not be real-valued and non-negative, and deduce that they are close in $L^2$ to a function $\chi(y) = \frac{1}{|H|^{1/2}} 1_{x + H}(y) e(P(y))$ where $P$ is some quadratic polynomial (in fact, they also prove a result for all $U^k$-norms and any LCA group).  When $f$ is real-valued and non-negative, it is clear that replacing $\chi$ by $|\chi|$ can only decrease the $L^2$ distance from $f$, so we deduce the above as an easy consequence.
\end{proof}


So, we may assume that $\eta$ is roughly the indicator function of a coset.  As we remarked earlier, this will be a proper coset, since $\eta$ has small support.  In particular, this means that most, but not all, of the mass of $\eta$ lies in some co-dimension one coset of $(\ZZ/2\ZZ)^d$.  Under these circumstances, it is always possible to obtain a modest energy increment, as the following lemma shows.

\begin{lemma}
  \label{lem:perturbative}
  Let $\phi \colon \ZZ^d \to \ZZ/2\ZZ$ be a surjective homomorphism.  For $i \in \{0,1\}$ write $A'_i = \phi^{-1}(\{i\})$, and write
  \[
    \alpha_i = (|A'_i|/|A'|)^{1/2} = \|A'_i\|_2 / \|A'\|_2
  \]
  so that $\alpha_0^2 + \alpha_1^2 = 1$.  Suppose that $\min\{\alpha_0, \alpha_1\} \le 0.1$.  Then for some $i \in \{0,1\}$ we have
  \[
    \frac{\cE(A'_i)}{\cE(A')} \ge (|A'| / |A'_i|)^{1/8} \ .
  \]
\end{lemma}
\begin{proof}
  The analysis is very similar to that of the previous section.  We define
  \[
    \gamma_i = \|A'_i\|_{U^3}
  \]
  for $i \in \ZZ/2\ZZ$.  By Lemma \ref{lem:gowers-cauchy-schwarz} we have that
  \[
    \|A'\|_{U^3} \le \|\gamma\|_{U^3} = \left(\gamma_0^8 + 14 \gamma_0^4 \gamma_1^4 + \gamma_1^8 \right)^{1/8} \ .
  \]
  Write $\gamma'_i = \gamma_i / \|A'\|_{U^3}$.  Then we may assume for contradiction that
  \[
    \gamma'_i / \alpha_i < \alpha_i^{-1/4}
  \]
  for each $i$, since otherwise the conclusion of the lemma is satisfied.  That is, we may assume $\gamma_i' < \alpha_i^{3/4}$.  Combining this with
  \[
    (\gamma'_0)^8 + 14 (\gamma'_0)^4 (\gamma'_1)^4 + (\gamma'_1)^8 \ge 1
  \]
  from above, we find that
  \[
    \alpha_0^6 + 14 \alpha_0^3 \alpha_1^3 + \alpha_1^6 > 1 \ .
  \]
  From this fact, together with $\alpha_0^2 + \alpha_1^2 = 1$, 
  we wish to deduce that $\min\{\alpha_0, \alpha_1\} > 0.1$.  This is now a routine algebraic exercise.  We set $t = \alpha_0^2 - \alpha_1^2$, which lies in the range $[-1,1]$, so that
  \[
    \left(\frac{1+t}{2}\right)^3 + 14 \left(\frac{1+t}{2}\right)^{3/2} \left(\frac{1-t}{2}\right)^{3/2} + \left(\frac{1-t}{2}\right)^3 > 1
  \]
  which, after expansion and rearrangement, yields
  \[
    2 + 6 t^2 + 14 (1 - t^2)^{3/2} > 8
  \]
  which further rearranges to
  \[
    (1 - t^2)^{3/2} > \frac{3}{7} (1 - t^2)
  \]
  and hence we deduce that 
  \[
    |t| < \frac{2 \sqrt{10}}{7}
  \]
  and so
  \[
    \min\{\alpha_0, \alpha_1\} = \left(\frac{1 - |t|}{2}\right)^{1/2} > 0.21
  \]
  as required.
\end{proof}

With this unenlightening computation complete, we formalize the discussion above to deduce Proposition \ref{prop:99pc-energy-case}.

\begin{proof}[Proof of Proposition \ref{prop:99pc-energy-case}]
  By Proposition \ref{prop:99pc-structure-thm} we may choose $\delta > 0$ such that for $\eta$ satisfying the condition of the statement, there is a subgroup $H \le (\ZZ/2\ZZ)^d$ and $x \in (\ZZ/2\ZZ)^d$ such that
  \[
    \left\| \eta / \|\eta\|_2 - \frac{1}{|H|^{1/2}} 1_{x+H} \right\|_2 \le 0.1 \ .
  \]

  We now argue that (WLOG) $H$ is proper.  Suppose for contradiction that it is not.  Once again, it is convenient to apply Proposition \ref{prop:discard-small-points}, now with parameter $R = 1/2$ (say); so that, either we have an energy increment at least as large as the one we require and are done, or we may decompose $\eta = \eta_{\bg} + \eta_{\sml}$ as in that statement.  By the triangle inequality, we deduce that
  \[
    \left\| \eta_{\bg} / \|\eta\|_2 - 2^{-d/2} \right\|_2 \le 0.1 + \|\eta_{\sml}\|_2 / \|\eta\|_2 \le 0.6 \ .
  \]
  However, since $\eta_{\bg}(y) / \|\eta\|_2$ is either $0$ or at least $2^{-12} K^{-4} \ge 2^4 \cdot 2^{-d/2}$, we have that
  \[
    \sum_{y \in (\ZZ/2\ZZ)^d} \left|\eta_{\bg}(y)/\|\eta\|_2 - 2^{-d/2}\right|^2 \ge \sum_{y \in (\ZZ/2\ZZ)^d} 2^{-d} = 1
  \]
  which is a contradiction.

  Hence, we may assume that $H$ is proper, and so there exists an index two subgroup $H \le H' < (\ZZ/2\ZZ)^d$.  We know that almost all of the $L^2$ mass of $\eta$ (or equivalently $A'$) is contained in just one of the cosets of $H'$.  Specifically, if $(y+H')$ is the coset \emph{not} containing $(x+H)$, then
  \[
    \| (\eta / \|\eta\|_2) \cdot 1_{y+H'}\|_2 \le \left\| \eta / \|\eta\|_2 - \frac{1}{|H|^{1/2}} 1_{x+H} \right\|_2 \le 0.1 \ .
  \]
  Hence, quotienting by $H'$ we obtain a map $\ZZ^d \to \ZZ/2\ZZ$ such that the hypotheses of Lemma \ref{lem:perturbative} apply.  This completes the proof.
\end{proof}

\subsection{Completing the proof of Theorem \ref{thm:main-theorem}}

It remains only to put all these pieces together to deduce Theorem \ref{thm:main-theorem}.

Initially, we are given a set $A \subseteq \ZZ$ such that $E(A) / |A|^3 \ge 1 / L$.  By Proposition \ref{prop:u3-energy-inequalities}, we deduce that the $U^3$ energy is also large, in that $\cE(A) \ge L^{-1/2}$.  

We now formally set up a recursive procedure for passing to subsets of $A$ with larger $U^3$ energy, stopping when the resulting set has small Fre\u{\i}man dimension.

Write $A_0 = A$, and write $K_i = 1 / \cE(A_i)$.  All our previous discussion will establish the following recursive fact.

\begin{lemma}
  \label{lem:recursive-lemma}
  Given a set $A_n \subseteq \ZZ$, either $\dim(A_n) \le 8 \log_2 K_n + 32$, or there exists a proper subset $A_{n+1} \subsetneq A_n$ such that $K_{n+1} \le K_n$ and
  \[
    \log_2\left(\frac{|A_n|}{|A_{n+1}|}\right) \ll \big(\log_2 K_0 + O(1)\big) \log_2\left(\frac{K_n}{K_{n+1}}\right) \ .
  \]
\end{lemma}

\begin{remark}
  Reviewing the preceding arguments, it is fairly clear that the constant $32$ appearing here could be replaced by $0.001$ (or indeed any fixed positive constant), at the expense of worsening the implied $O(1)$ constants elsewhere.  This means that the bound on the Fre\u{\i}man dimension of $8 \log_2(K) + 0.001$ we obtain this way is essentially best possible in some sense, at least if we are allowed to replace the indicator function $A$ with a (Gaussian) density.

  The reason we do not stress this optimality elsewhere in this paper is that it only holds with reference to the $U^3$ energy.  When we attempt to transfer this result to other measures of structure -- such as $U^2$ energy or doubling -- even the leading constant in $O(\log K)$ is no longer optimal.
\end{remark}

\begin{proof}[Proof of Theorem \ref{thm:main-theorem} assuming Lemma \ref{lem:recursive-lemma}]
  Since each inclusion is proper, the process must terminate.  That is, there is some $N$ such that $A_N$ is defined and
  \[
    \dim(A_N) \le 8 \log_2 K_N + 32 \le 8 \log_2 K_0 + 32 \le 4 \log_2 L + 32 \ .
  \]
  By telescoping the recursive size inequality, we deduce that
  \[
    \log_2\left(\frac{|A_0|}{|A_{N}|}\right) \ll \big(\log_2 K_0 + O(1)\big) \log_2\left(\frac{K_0}{K_{N}}\right)
  \]
  and bounding $K_N \ge 1$ trivially we deduce
  \[
    |A_0|/|A_N| \ll \exp\big(O(\log^2 K_0) \big) \ll \exp\big(O(\log^2 L )\big)
  \]
  as required.
\end{proof}

\begin{proof}[Proof of Lemma \ref{lem:recursive-lemma}]

  Let $d = \dim(A_n)$.  By the definition of Fre\u{\i}man dimension, we may find a subset $A_n' \subseteq \ZZ^d$ and a Fre\u{\i}man ($2$-)isomorphism $\phi \colon A_n \leftrightarrow A_n'$.

  We remark the following.

  \begin{proposition}
    \label{prop:u3-freiman}
    The $U^3$ energy is preserved under Fre\u{\i}man isomorphisms.  In the current setting, this means that $\|A_n\|_{U^3} = \|A_n'\|_{U^3}$.
  \end{proposition}
  \begin{proof}
    It suffices to show that a Fre\u{\i}man homomorphism sends $3$-dimensional parallelepipeds to $3$-dimensional parallelepipeds.  Certainly this holds for $2$-dimensional parallelograms, also known as additive quadruples: indeed, this is exactly the definition of a Fre\u{\i}man homomorphism.

    We observe that for an abelian group $G$, a configuration $\{0,1\}^3 \to G$ is a $3$-dimensional parallelepiped (that is, has the form $\omega \mapsto g_0 + \omega_1 g_1 + \omega_2 g_2 + \omega_3 g_3$) if and only if every codimension one face of $\{0,1\}^3$ (that is, sets of the form $\{ \omega \in \{0,1\}^d \colon \omega_i = r \}$ for fixed $i \in \{1,2,3\}$ and $r \in \{0,1\})$ maps to an additive quadruple.  The proof of this is an easy exercise, and it clearly implies the result.
  \end{proof}

  We now invoke one of Proposition \ref{prop:small-energy} and Proposition \ref{prop:99pc-energy-case}: which depends on whether the quantity $\cE(\eta)$ is bigger or smaller than $(1 - \delta)$, where $\delta$ is the positive absolute constant appearing in Proposition \ref{prop:99pc-energy-case}. We set $A'_{n+1}$ to be the set $A'_S$ produced by either of those results as appropriate, and define $A_{n+1} = \phi^{-1}(A'_{n+1})$.
  
  We conclude that either
  \[
    K_n / K_{n+1} \ge (|A_n| / |A_{n+1}|)^{1/8}
  \]
  as implied by either the first case of Proposition \ref{prop:small-energy} or by Proposition \ref{prop:99pc-energy-case}; or,
  \[
    K_{n} / K_{n+1} \ge 1 + \delta/2
  \]
  and
  \[
    |A_n| / |A_{n+1}| \le 2^{24} K_n^{16} \delta^{-8} \le 2^{24} K_0^{16} \delta^{-8}
  \]
  as returned by the second case of Proposition \ref{prop:small-energy}.  Taking logarithms, we find that either
  \[
    \log_2\left(\frac{|A_n|}{|A_{n+1}|}\right) \le 8 \log_2\left(\frac{K_n}{K_{n+1}}\right)
  \]
  or
  \begin{align*}
    \log_2\left(\frac{|A_n|}{|A_{n+1}|}\right) &\le 24 + 8 \log(1/\delta) + 16 \log K_0 \\
                                               &\le \frac{24 + 8 \log(1/\delta) + 16 \log K_0}{\log(1 + \delta/2)} \log\left(\frac{K_n}{K_{n+1}}\right)
  \end{align*}
  either of which is acceptable for Lemma \ref{lem:recursive-lemma}.
\end{proof}

\begin{remark}
  The only bar to obtaining completely explicit constants in Theorem \ref{thm:main-theorem} is currently the mysterious constant $\delta$ obtained from the structure theorem for functions with almost maximal $U^3$ energy.  Doubtless this constant is moderately sensible; however, no explicit dependencies are given in \cite{tao-eisner} and one would have to examine their proofs, which in turn rely on the classification of almost extremizers in Young's inequality due to Fournier \cite{fournier}.

  Alternatively, one might reasonably seek a self-contained proof of Proposition \ref{prop:99pc-structure-thm}, exploiting the fact that we only require it for non-negative functions, rather than going via \cite{tao-eisner}.

  Given we have made no attempt to optimize constants in any part of this argument, actually pursuing this would perhaps be in poor taste. The purpose of this comment is to reassure the reader that the implied constants in Theorem \ref{thm:main-theorem} are not too ridiculous.
\end{remark}

\appendix

\section{Alternative proof of an inequality of Shkredov}
\label{sec:u3-inequality}

We now give, for interest only, an alternative proof of the harder direction of Proposition \ref{prop:u3-energy-inequalities}, a result due to Shkredov \cite{shkredov}.  Since this argument is orthogonal to the main theme of the paper, we place it separately here.

Recall we are trying to show the following.

\begin{proposition}
  \label{prop:hard-energy-inequality}
  If $Z$ is any abelian group and $f \colon Z \to \RR_{\ge 0}$ a non-negative finitely supported function, then
  \[
    \| f \|_{U^3} \ge \| f \|_{U^2}^2 / \|f\|_1 \ .
  \]
\end{proposition}

When $f = 1_X$ is the indicator function of a set, this can be interpreted as follows: if $X$ contains many $2$-dimensional parallelepipeds -- say, $\delta |X|^3$ of them -- then it contains many $3$-dimensional parallelepipeds, i.e.~at least $\delta^4 |X|^4$.  A result of this flavour is implicit in \cite{gt-equivalence}, although the implied proof is quite complicated (involving the Balog--Szemeredi--Gowers theorem and Fre\u\i{}man modelling) and obtains a worse exponent.

We note that Proposition \ref{prop:hard-energy-inequality} does \emph{not} follow from the elementary fact of monotonicity of the Gowers norms, usually stated as $\|f\|_{U^3} \ge \|f\|_{U^2}$.  Indeed, this formula assumes probability measure on $Z$ is being used (and so is false with our conventions); to correct this requires a factor of the measure $\mu(Z)$ of the whole group, which renders the result useless unless $f$ is dense.  In our case, $Z$ could even be infinite and so $f$ far from dense.

\begin{proof}[Proof of Proposition \ref{prop:hard-energy-inequality}]
  Define a function
  \begin{align*}
    R \colon Z \times Z &\to \RR_{\ge 0} \\
                  (a,b) &\mapsto \sum_{x \in Z} f(x) f(x+a) f(x+b) f(x+a+b)
  \end{align*}
  or in other words, the number of parallelepipeds with sides $a$ and $b$, counted weighted by $f$.  We note that
  \begin{equation}
    \label{rprop1}
    \sum_{a,b \in Z} R(a,b) = \|f\|_{U^2}^4
  \end{equation}
  and
  \begin{equation}
    \label{rprop2}
    \sum_{a,b \in Z} R(a,b)^2 = \|f\|_{U^3}^8 \ .
  \end{equation}
  We also have that
  \[
    \sum_{b \in Z} R(a, b) = \left(f \ast f_{-} (a) \right)^2
  \]
  where we write $f_{-}$ for the function $f_{-}(x) = f(-x)$; and since $\sum_a f \ast f_{-}(a) = \|f\|_1^2$, we deduce that
  \begin{equation}
    \label{rprop3}
    \sum_{a \in Z} \left( \sum_{b \in Z} R(a,b) \right)^{1/2} = \|f\|_1^2 \ .
  \end{equation}

  One might hope that in fact \eqref{rprop1}, \eqref{rprop2} and \eqref{rprop3} are the only facts we need to bear in mind to prove the result: that is, that we can forget where $R$ came from and just treat it as an arbitrary non-negative function on $Z \times Z$.  In other words, it would suffice to show the following purely analytic result.

  \begin{lemma}
    \label{key-lemma}
    Let $R \colon Z \times Z \to \RR_{\ge 0}$ be an arbitrary function.  Then
    \[
      \sum_{a, b \in Z} R(a,b)^2 \ge \frac{\left(\sum_{a,b \in Z} R(a,b) \right)^4}{\left[ \sum_{a \in Z} \left(\sum_{b \in Z} R(a,b) \right)^{1/2} \right]^2 \left[ \sum_{b \in Z} \left(\sum_{a \in Z} R(a,b) \right)^{1/2} \right]^2 } \ .
    \]
  \end{lemma}

  Fortunately this is true.  Although Lemma \ref{key-lemma} is essentially completely elementary, we are not aware of a completely elementary proof.  Instead, we rely on some standard entropy inequalities.

  \begin{proof}[Proof of Lemma \ref{key-lemma}]
  Define
  \[
    W(a,b) := R(a,b) \left(\sum_{a,b \in Z} R(a,b) \right)^{-1}
  \]
  i.e.~normalizing so that $\sum_{a,b} W(a,b) = 1$.
  We may think of $W(a,b)$ as the joint probability distribution of a pair of random variables $X, Y$ on $Z$, so that $\PP[X=x, Y=y] = W(x,y)$.  Then, the quantities
  \[
    W_X(a) = \sum_{ b \in Z} W(a,b)
  \]
  and
  \[
    W_Y(b) = \sum_{ a \in Z} W(a,b)
  \]
  are precisely the marginal distributions of $X$ and $Y$.

  What we want to prove is then that
  \[
    \sum_{a,b \in Z} W(a,b)^2 \ge \left(\sum_{a \in Z} W_X(a)^{1/2} \right)^{-2} \left( \sum_{b \in Z} W_Y(b)^{1/2} \right)^{-2} \ .
  \]
  This statement is of the following flavour: ``if $X$ and $Y$ are concentrated on a few values then so is $(X,Y)$''.  A more usual way to express the same sentiment might be in terms of the entropy inequality $H(X,Y) \le H(X) + H(Y)$.  But noting the inequalities (which are standard under the title ``monotonicity of R\'enyi entropy'', or are easy convexity arguments)
  \[
    2 \log \left[ \sum_a W_X(a)^{1/2} \right] \ge H(X)
  \]
  and similarly for $Y$; and
  \[
    \log \left[ \sum_{a,b} W(a,b)^2\right] \ge - H(X,Y)
  \]
  this implies the result.
  \end{proof}
  This completes the proof of Proposition \ref{prop:hard-energy-inequality}.
\end{proof}

\section{A result about doubling and reduction modulo $2$}
\label{app:doubling}

We will record in full the proof of the following result, even though we will not need it in the main argument.

\begin{proposition}
  \label{prop:few-values-mod-2}
  Suppose $A \subseteq \ZZ^d$ is a finite set, and let $|A+A| = K |A|$.

  Consider the image of $A$ under reduction modulo $2$; that is,
  \[
    \tilde{A} := \{ x \bmod 2 \colon x \in A \} \subseteq (\ZZ/2\ZZ)^d \ .
  \]
  Then $|\tilde{A}| \le K^6$.
\end{proposition}

Our key ingredient will be a well-known covering argument due to Ruzsa \cite{ruzsa-covering}, which we state in its most general form.

\begin{lemma}
  Let $X$, $Y$ be finite subsets of the same abelian group $Z$.  

  Then there exists an integer $k$, $k \le |X+Y|/|X|$, and elements $\{y_1,\dots,y_k\} \subseteq Y$ such that
  \[
    Y \subseteq \bigcup_{i=1}^k (y_i + X - X) \ .
  \]
\end{lemma}
We briefly recall the proof.
\begin{proof}
  The trick is to select a maximal set $\{y_1, \dots, y_k\} \subseteq Y$ such that the sets $(y_i + X)$ are disjoint.  By maximality, for every $y \in Y$ there exists $i \in [k]$ such that $(y + X) \cap (y_i + X) \ne \emptyset$.  Equivalently, $y \in y_i + X - X$.  Furthermore, note that $(y_i + X)$ are disjoint subsets of $X+Y$, each of size $|X|$, so clearly $k |X| \le |X+Y|$ as required.
\end{proof}

We deduce the following, which is a slight variant on the well-known result that if $A$ has small doubling then $A-A$ is an approximate group.

\begin{corollary}
  \label{cor:somewhat-convex}
  Suppose $A$ is a finite subset of an abelian group $Z$, and $|A+A| = K |A|$.  Then $(2A-2A)$ is covered by at most $K^6$ translates of $2 \cdot (A - A)$.
\end{corollary}
Note by $2 \cdot X$ we mean the \emph{dilate} $\{ 2 x \colon x \in X\}$, as opposed to the sumset $2 X = X + X$.
\begin{proof}
  We apply the Ruzsa covering lemma with $X = 2\cdot A$ and $Y = 2A-2A$.  We deduce that $Y$ is covered by $k$ translates of $2 \cdot A - 2 \cdot A = 2 \cdot(A-A)$, where
  \[
    k \le |2 \cdot A + 2A - 2A|/|A| \le |2A + 2A - 2A|/|A| \le K^6
  \]
  where we have used the obvious inclusion $2 \cdot A \subseteq 2 A$ and Pl\"{u}nnecke's inequalities (see e.g.~\cite[Corollary 6.29]{tao-vu}).
\end{proof}

\begin{remark}
  Our preferred way of thinking about this result is that $A-A$ satisfies some kind of rough convexity.  Indeed, for a convex subset $X \subseteq \RR^d$ we have $2 \cdot X = 2 X$.  It is fairly clear this cannot happen for non-trivial subsets of $\ZZ^d$, since there is a parity obstruction to writing $x_1 + x_2 = 2 x_3$ for general $x_1,x_2,x_3 \in X$.  However, if $X \subseteq \ZZ^d$ were the intersection of $\ZZ^d$ with a convex subset of $\RR^d$, we might expect the slight weakening $2 X \subseteq 2 \cdot X + \{-1,0,1\}^d$ to hold.

  It is this property that the previous corollary seeks to emulate.
\end{remark}

Our original proposition is a trivial consequence of Corollary \ref{cor:somewhat-convex}.

\begin{proof}[Proof of Proposition \ref{prop:few-values-mod-2}]
  By Corollary \ref{cor:somewhat-convex} we have that $2A - 2A \subseteq 2 \cdot (A - A) + S$ for some set $S$, $|S| \le K^6$.  Therefore $2A - 2A$ takes at most $K^6$ distinct values modulo $2$.  Equivalently, $|2 \tilde{A} - 2 \tilde{A}| \le K^6$.  It follows trivially that $|\tilde{A}| \le K^6$.
\end{proof}

\begin{remark}
  The same argument shows that the number of residues modulo $N$ is bounded by $K^{4+N}$ for any positive integer $N$.
\end{remark}

\bibliography{master}{}
\bibliographystyle{halpha}
\end{document}